\documentclass[12pt]{amsart}
\usepackage{fullpage,xcolor,graphicx}
\usepackage{concmath}
\usepackage[T1]{fontenc}
\usepackage{hyperref}
\usepackage{pifont}
\usepackage{amsthm}
\usepackage{amsfonts}
\usepackage{amssymb}
\usepackage[mathscr]{euscript}
\usepackage[all]{xy}
\usepackage{amsmath}
\usepackage{epsfig}

\newtheorem{theorem}{Theorem}

\newtheorem{proposition}[theorem]{Proposition}

\theoremstyle{remark}
\newtheorem{definition}{Definition}[section]

\newtheorem*{acknowledgments}{Acknowledgments}

\theoremstyle{remark}

\begin{document}
\title{Band-Passes and Long Virtual Knot Concordance}
\author{Micah Chrisman}
\address{
Department of Mathematics\\
Monmouth University \\
West Long Branch, NJ 07764 \\
USA
}
\email{mchrisma@monmouth.edu}


\begin{abstract} Every classical knot is band-pass equivalent to the unknot or the trefoil. The band-pass class of a knot is a concordance invariant. Every ribbon knot, for example, is band-pass equivalent to the unknot. Here we introduce the long virtual knot concordance group $\mathscr{VC}$. It is shown that for every concordance class $[K] \in \mathscr{VC}$, there is a $J \in [K]$ that is not band-pass equivalent to $K$ and an $L \in [K]$ that is not band-pass equivalent to either the long unknot or any long trefoil.  This is accomplished by proving that $v_{2,1}+v_{2,2} \pmod 2$ is a band-pass invariant but not a concordance invariant of long virtual knots, where $v_{2,1}$ and $v_{2,2}$ generate the degree two Polyak group for long virtual knots.  
\end{abstract}
\maketitle

Band-pass equivalence of knots is generated by isotopy and the moves depicted in Figure \ref{fig_band_pass}. Every classical knot is band-pass equivalent to either the unknot or the right trefoil. In particular, every ribbon knot is band-pass equivalent to the unknot. Furthermore, the band-pass class of a knot is a concordance invariant. The band-pass class of a knot is completely determined by the Arf invariant. The Arf invariant is essentially combinatorial.  It can be defined as the coefficient $c_2(K)$ of $z^2$ in the Conway polynomial of $K$, modulo two. This classic story, due to Kauffman \cite{on_knots}, elegantly displays the interaction of topology and combinatorics in knot theory. The topological equivalence relation of \emph{concordance} is a subset of the combinatorial relation of \emph{band-pass equivalence}.

\begin{figure}[htb]
\begin{tabular}{|ccc|ccc|} \hline
 & & & & & \\
\begin{tabular}{c}
\def\svgwidth{.5in}
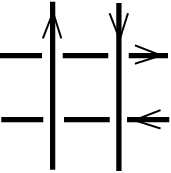 \end{tabular} & $\longleftrightarrow$ &  \begin{tabular}{c} \def\svgwidth{.5in}
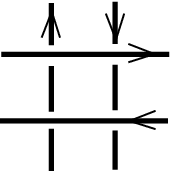 \end{tabular} & \begin{tabular}{c} \def\svgwidth{.5in}
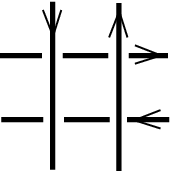 \end{tabular} &  $\longleftrightarrow$ & \begin{tabular}{c} \def\svgwidth{.5in}
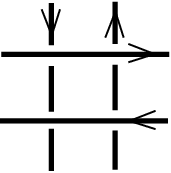 \end{tabular} \\  & & & & & \\ \hline
\end{tabular}
\caption{The two versions of the band-pass move.} \label{fig_band_pass}
\end{figure}

Virtual knot theory can be roughly defined as the study of properties of knot diagrams that are preserved under isotopy and stabilization of the ambient space. Virtual knots are thus partly combinatorial and partly topological. From this point of view, it is natural to reconsider the story of band-pass equivalence for concordance of virtual knots. This line of investigation was originally suggested by Kauffman \cite{lou_cob}. Here we consider long virtual knots. The main theorem is that every concordance class of long virtual knots contains representatives that are not band-pass equivalent to each other and representatives that are not band-pass equivalent to either a long unknot or a long trefoil. 
\newline
\newline
The motivation to study long virtual knots comes from two observations. In  Section \ref{sec_concordance}, we will show that long virtual knots form a group under concatenation. Thus long virtual knots generalize the well-known classical knot concordance group. The second observation is that the Arf invariant has a simple extension to long virtual knots.  With a choice of base point, $c_2(K)=v_{2,1}(K)=v_{2,2}(K)$ \cite{ckr} for all classical knots $K$, where $v_{2,1}, v_{2,2}$ generate the Goussarov-Polyak-Viro finite-type invariants of degree two for long virtual knots \cite{GPV}. Neither $v_{2,1} \pmod{2}$ nor $v_{2,2} \pmod{2}$ is a band-pass or concordance invariant for long virtual knots by themselves. However, it will be shown that $v_{2,1}+v_{2,2} \pmod{2}$ is a band-pass invariant but is not an invariant of virtual knot concordance. The main theorem is a direct consequence of this result.
\newline
\newline
An outline of this paper is as follows. Section \ref{sec_concordance} begins with a review of virtual knot concordance. A combinatorial proof that long virtual knots form a group is also given in this section. Section \ref{sec_band} is devoted to the band-pass invariant and the proof of the main theorem. Section \ref{sec_future} contains further discussion and indicates directions for future research.

\section{The Long Virtual Knot Concordance Group} \label{sec_concordance}

\subsection{Concordance}  Long virtual knot diagrams are immersions $\mathbb{R} \to \mathbb{R}^2$, identical with the $x$-axis outside a compact set, such that each self-intersection is a transversal double point. Each double point is marked as a classical or virtual crossing. These diagrams are considered equivalent up to the \emph{extended Reidemeister moves} \cite{KaV} (see Figure \ref{fig_rmoves}). All long knots are oriented from $-\infty \to \infty$. A \emph{long virtual link} is a virtual link in the usual sense except that exactly one component is a long virtual knot.  Long virtual knots were first introduced by Goussarov-Polyak-Viro \cite{GPV} in their study of finite-type knot invariants.

\begin{figure}[htb]
\begin{tabular}{|ccc|ccc|ccc|} \hline
 & & & & & & & & \\
\begin{tabular}{c}  \def\svgwidth{.5in} 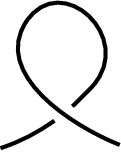 \end{tabular} & \begin{tabular}{c} $\leftrightharpoons$ \end{tabular} & 
\begin{tabular}{c}  \def\svgwidth{.5in} 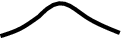 \end{tabular} & \begin{tabular}{c}  \def\svgwidth{.52in} 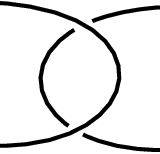 \end{tabular} & \begin{tabular}{c} $\leftrightharpoons$ \end{tabular} & 
\begin{tabular}{c}  \def\svgwidth{.52in} 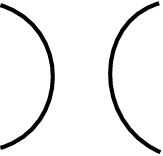 \end{tabular} & \begin{tabular}{c}  \def\svgwidth{.5in} 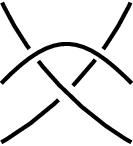 \end{tabular} & \begin{tabular}{c} $\leftrightharpoons$ \end{tabular} & 
\begin{tabular}{c}  \def\svgwidth{.5in} 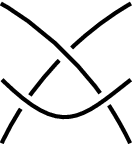 \end{tabular}   \\  
 & $\Omega 1$ & &  & $\Omega 2$ & & &$\Omega 3$ & \\\hline
 & & & & & & & & \\
\begin{tabular}{c}  \def\svgwidth{.5in} 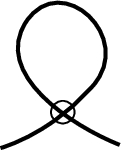 \end{tabular} & \begin{tabular}{c} $\leftrightharpoons$ \end{tabular} &  \begin{tabular}{c}  \def\svgwidth{.5in} \input{r1_R.eps_tex} \end{tabular} & \begin{tabular}{c}  \def\svgwidth{.52in} 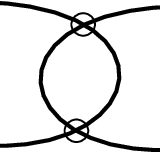 \end{tabular} & \begin{tabular}{c} $\leftrightharpoons$ \end{tabular} & 
\begin{tabular}{c}  \def\svgwidth{.52in} \input{r2_R.eps_tex} \end{tabular} & \begin{tabular}{c}  \def\svgwidth{.5in} 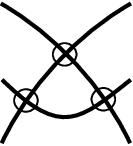 \end{tabular} & \begin{tabular}{c} $\leftrightharpoons$ \end{tabular} & 
\begin{tabular}{c}  \def\svgwidth{.5in} 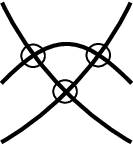 \end{tabular}   \\  
 & $v\Omega 1$ & &  & $v\Omega 2$ & & & $v\Omega 3$ & \\ \hline
\multicolumn{9}{|c|}{} \\
\multicolumn{9}{|c|}{\begin{tabular}{ccc}\begin{tabular}{c}  \def\svgwidth{.5in} 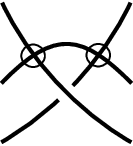 \end{tabular} & \begin{tabular}{c} $\leftrightharpoons$ \end{tabular} & \begin{tabular}{c}  \def\svgwidth{.5in} 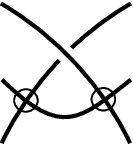 \end{tabular} \\ & $v\Omega 4$ & \end{tabular}} \\ \hline
\end{tabular} \caption{The extended Reidemeister moves.}\label{fig_rmoves}
\end{figure}

Before discussing virtual knot concordance, it is prudent to review concordance of classical knots. Let $K_0,K_1$ be oriented knots in $\mathbb{S}^3$. Then $K_0$, $K_1$ are said to be (smoothly) \emph{cobordant} if there is a compact, connected, oriented, and smooth surface $S$ that is properly embedded in $\mathbb{S}^3 \times [0,1]$ so that $\partial S=K_1 \sqcup -K_0$ and $S \cap \mathbb{S}^3 \times \{i\}=(-1)^{i+1} K_i$, for $i=0,1$. Here $-K$ denotes $K$ with the opposite orientation. If $S$ may be chosen to be an annulus $\mathbb{S}^1 \times [0,1]$, then $K_0$ and $K_1$ are said to be (smoothly) \emph{concordant}. 
\newline
\newline
The requirement that the surface $S$ be smooth allows for a combinatorial perspective of cobordism. Let $\mu:S \to [0,1]$ denote the inclusion $S \to \mathbb{S}^3 \times [0,1]$ followed by the projection to $[0,1]$. By Morse theory, $\mu$ may be chosen so that all the critical points appear as local maximums, local minimums, or saddles (see Figure \ref{fig_morse_crit}). The critical points of $\mu$ may further be assumed to all have different critical values. For regular values $t \in [0,1]$, $S \cap \mathbb{S}^3 \times \{t\}$ is an oriented link in $\mathbb{S}^3$. In a neighborhood of a local extremum, we see either the \emph{birth} or \emph{death} of a small unknot disjoint from all other components of the link (see Figure \ref{fig_concordance}, left). In a neighborhood of a saddle, we see a transition from one type of hyperbola to another, as in Figure \ref{fig_concordance}, right.  In between the critical points, the link $S \cap \mathbb{S}^3 \times \{t\}$ undergoes an ambient isotopy in $\mathbb{S}^3$. A cobordism can thus be visualized as a \emph{movie} consisting of a finite sequence of births, deaths, saddles, and Reidemeister moves.
\newline

\begin{figure}
\begin{tabular}{|ccc|} \hline
 & & \\
\begin{tabular}{c}  \def\svgwidth{1in} 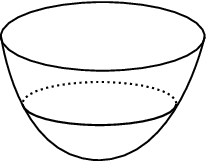 \end{tabular} & \begin{tabular}{c}  \def\svgwidth{1in} 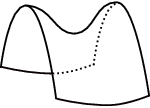 \end{tabular} & \begin{tabular}{c}  \def\svgwidth{1in} 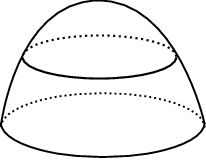 \end{tabular} \\ 
``death'' & ``saddle'' & ``birth'' \\ \hline
\end{tabular} \caption{Critical points of Morse functions.}\label{fig_morse_crit}
\end{figure}

To the map $\mu:S\to [0,1]$, we may associate its \emph{Reeb graph} $\Gamma_{\mu}$. The Reeb graph is the $1$-complex obtained by identifying to a point each path component of each level set $\mu^{-1}(t)$. The vertices of $\Gamma_{\mu}$ are thus in one-to-one correspondence with the set of critical points of $\mu$ together with two points representing $K_0$ and $K_1$. Each edge corresponds to a path through time of a link component. Every univalent vertex of $\Gamma_{\mu}$ corresponds to either $K_0$, $K_1$, a birth, or a death. Each trivalent vertex of $\Gamma_{\mu}$ is a saddle. See Figure \ref{fig_reeb}.

\begin{figure}[htb]
\begin{tabular}{|c|c|}  \hline
\begin{tabular}{c} \\
\def\svgwidth{2in}
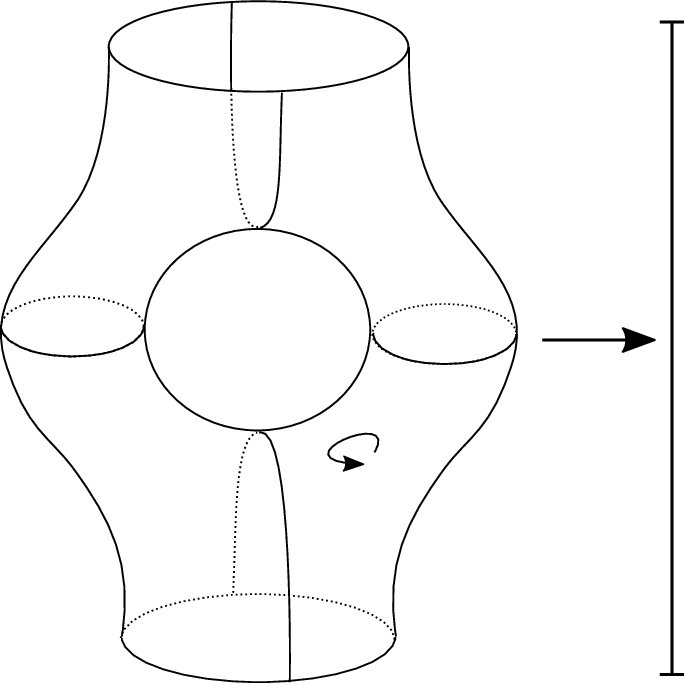 
\end{tabular}
& 
\begin{tabular}{c}
\def\svgwidth{.53in}
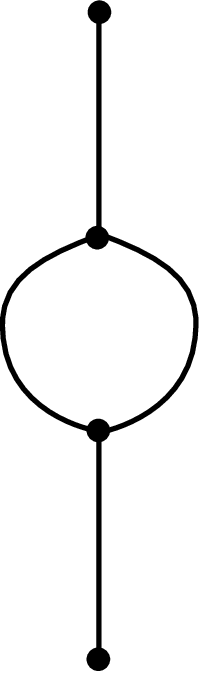 
\end{tabular} \\ \hline
\end{tabular} 
\caption{Schematic of a cobordism (left) and its Reeb graph (right).}\label{fig_reeb}
\end{figure}

It follows that  $\Gamma_{\mu}$ can be reconstructed from the cobordism movie of births, deaths, saddles, and Reidemeister moves. The vertices of $\Gamma_{\mu}$ are as described above. Each knot appearing in the movie is an edge of $\Gamma_{\mu}$. An edge is incident to a vertex if its knot is involved in a saddle, birth, or death at the corresponding critical point (or if it is $K_0$ or $K_1$). For a concordance of knots, the cobordism surface must be an annulus. Also note that the Euler characteristic can be computed from the critical point data of births $b$, deaths $d$, and saddles $s$. A concordance of knots is therefore identified with a movie that begins with $K_0$, ends with $K_1$, has a connected Reeb graph, and satisfies $\#b-\#s+\#d=0$.
\newline
\newline
This motivates the definition of concordance for closed and long virtual knots. The combinatorial formulation used here is adapted from Kauffman \cite{DKK,lou_cob}. Earlier topological formulations are due to Carter-Kamada-Saito \cite{CKS} and Turaev \cite{turaev_cobordism} (see also Fenn-Rourke \cite{fenn_rack}). The pairwise equivalence of these definitions follows from \cite{CKS}.

\begin{definition}[Concordance] \label{defn_concordance} Two (long or closed) virtual knots $K_0$ and $K_1$ are said to be \emph{concordant} if they may be obtained from one another by a finite sequence of extended Reidemeister moves, births $b$, deaths $d$, and saddle moves $s$.  The sequence must have a connected Reeb graph and satisfy $\#b-\#s+\#d=0$. If $K_0$ and $K_1$ are concordant, we write $K_0 \asymp K_1$. 
\end{definition}

\begin{figure}[htb]
\begin{tabular}{|c|c|}  \hline
\begin{tabular}{c} \\
$K \sqcup \begin{array}{c} \def\svgwidth{.35in}
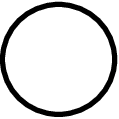 \end{array} $\\ \\ \ $\text{birth} \uparrow \,\,\,\, \downarrow \text{death}$ \\ \\
 $K$
\end{tabular}
& 
\begin{tabular}{ccc}
\multicolumn{3}{c}{\underline{saddle move}} \\ & &\\
\begin{tabular}{c}
\def\svgwidth{.73in}
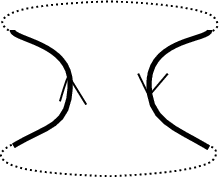 
\end{tabular}
& $\leftrightarrow$ &
\begin{tabular}{c}
\def\svgwidth{.73in}
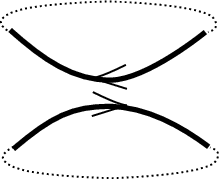 
\end{tabular}  \end{tabular} \\ \hline
\end{tabular} 
\caption{Births, Deaths, and Saddles.}\label{fig_concordance}
\end{figure}

\textbf{Example:} A concordance movie  beginning with the ``fly'', $F$, is depicted in Figure \ref{fig_conc_ex}. The red arc indicates a saddle move. $F$ is well-known to be a non-trivial long virtual knot (cf. \cite{quatlong}). The Reeb graph of this movie is given in on the far right of Figure \ref{fig_conc_ex}.

\begin{figure}[htb]
\begin{tabular}{|c||c||c||c|c|} \hline 
\begin{tabular}{c} \def\svgwidth{1in}
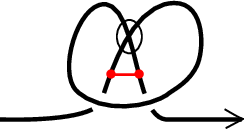 \end{tabular} &  \begin{tabular}{c} \def\svgwidth{1in}
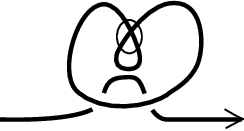 \end{tabular} & \begin{tabular}{c} \def\svgwidth{1in}
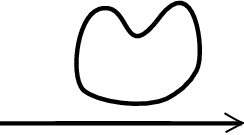 \end{tabular} & \begin{tabular}{c} \\  \\ \def\svgwidth{1in}
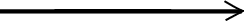 \end{tabular} & \begin{tabular}{c} \def\svgwidth{.75in}
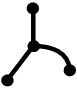 \end{tabular} \\ \hline \end{tabular} 
\caption{(Left) The ``fly'', $F$ is concordant to the unknot. (Right) The Reeb graph of the movie on the left.} \label{fig_conc_ex}
\end{figure}

\subsection{The Long Virtual Knot Concordance Group}Long virtual knots form a monoid under the operation of \emph{concatenation}: $K_1 \# K_2$ means draw $K_2$ to the right of $K_1$. The unit, denoted $1$, is the long knot given by the $x$-axis. It is well known that concatenation is not commutative \cite{quatlong, manturov_compact_long}. It follows immediately from Definition \ref{defn_concordance} that concatenation is well defined on concordance classes: if $L_1 \asymp R_1$ and $L_2 \asymp R_2$, then $L_1 \# L_2 \asymp R_1 \# R_2$.
\newline
\newline
The long virtual knot concordance group $\mathscr{VC}$ is defined by analogy to the classical knot case (see \cite{on_knots} for the classical definition). The elements are concordance classes, the identity is $1$, the operation is $\#$, and inverses are defined as follows. Let $r(K)$ denote the reflection of the diagram $K$ about a vertical line $l$ intersecting $K$ in one point to the right of a closed $2$-ball that contains all points of $K$ not on the $x$-axis. Let $-r(K)$ denote the change of orientation, so that $-r(K)$ is oriented from $-\infty \to \infty$. Then $K^{-1}:=-r(K)$. 

\begin{figure}[htb]
\begin{tabular}{|c|c|} \hline
\begin{tabular}{c}
\def\svgwidth{2.5in}
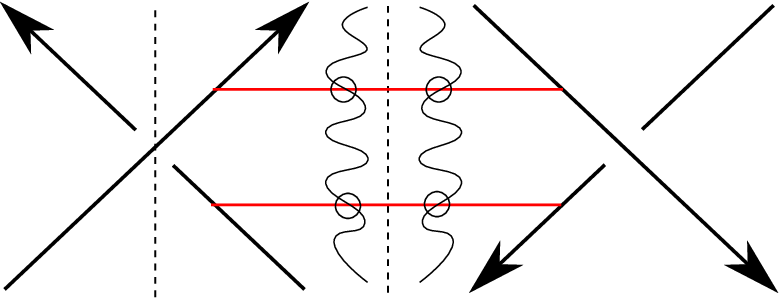 \end{tabular} & \begin{tabular}{c} \def\svgwidth{2.5in}
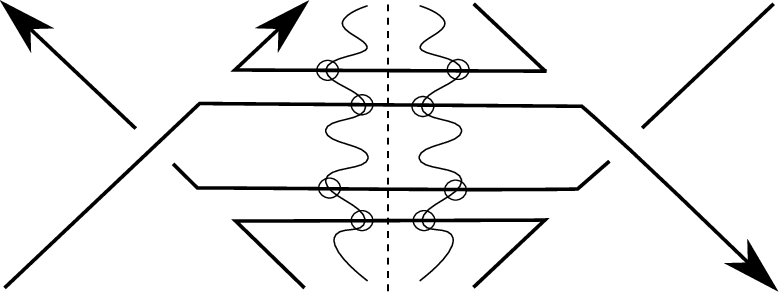 \end{tabular} \\ \hline  
\end{tabular}
\caption{Two saddle moves used in the proof of Theorem 1.} \label{fig_inverse}
\end{figure}
\begin{theorem} $(\mathscr{VC},\#,1)$ is a group. 
\end{theorem}
\begin{proof} We will show $K \# -r(K) \asymp 1$. Any two long virtual knots having the same Gauss diagram on $\mathbb{R}$ are equivalent by moves $v\Omega 1-v\Omega 4$. We may thus assume that $K$ is given so that all classical crossings $z$ lie on a line $c$ perpendicular to the $x$-axis. The classical crossings are connected together by arcs to produce the correct Gauss diagram of $K$, inserting virtual crossings where necessary. Furthermore assume that the the line $l$ defining $K\#-r(K)$ is to the right of $c$. For each $z$, let $m(z)$ denote the corresponding crossing in $-r(K)$. For each $z$, draw a pair of parallel red arcs near the crossing point as in Figure \ref{fig_inverse}.
\newline
\newline
It may be assumed that all intersections of the red arcs with $K \# -r(K)$ are transversal and that the arcs are symmetric with respect to $l$. Apply two saddle moves on the red arcs. This gives Figure \ref{fig_inverse}, right, locally for each $z$. All crossings coming from the red arcs with $K \# -r(K)$ are virtual.  Let $L$ be the resulting long virtual link.
\newline
\newline
\emph{Observation:} The cobordism given in Figure \ref{fig_inverse} adds two components to the long virtual link diagram. To see this, begin constructing $L$ with some chosen crossing $z$ of $K$. Perform one of the two saddle moves on the red arcs. By Figure \ref{fig_concordance}, a saddle move beginning with two arcs from the same component increases the number of components by one. By symmetry, the endpoints of the other red arc for $z$ and $m(z)$ must now lie in the same component. Surgering along the second red arc thus also increases the number of components by one. Proceed with crossings $y \ne z$ of $K$. 
\newline
\newline
\emph{Observation:} For each $z$, the pair of crossings $z,m(z)$ of $L$ can be removed with a $\Omega 2$ move. All classical crossings of $L$ are annihilated in this way.
\newline
\newline
Thus $L$ is equivalent to a long virtual link having no classical crossings. It has $1$ as a component together with $2n$ disjoint circles, where $n$ is the number of classical crossings of $K$. Kill each of these disjoint circles to obtain a cobordism to $1$. Since the Reeb graph is connected and $\#b -\#s+\#d=-2n+2n=0$, it is a concordance.
\end{proof}

\section{The Band-Pass Invariant} \label{sec_band}

Two long virtual knots will be said to be \emph{band-pass equivalent} if they are related by a finite sequence of extended Reidemeister moves and band-pass moves (see Figure \ref{fig_band_pass}). Here we will prove that, contrary to the classical case, the virtual knot concordance relation is not a subset of the band-pass relation. To do this, it is necessary to find a band-pass invariant that is not a concordance invariant. It will be shown that a simple modification of the Arf invariant satisfies this criterion. We begin by reviewing the definition of the Arf invariant of a classical knot in terms of the degree two Goussarov-Polyak-Viro finite-type invariants $v_{2,1}$ and $v_{2,2}$ (see \cite{GPV}).
\newline
\newline
Let $D_{K}$ be a Gauss diagram of a long virtual knot $K$. A \emph{subdiagram} $D'$ of $D_K$ is a diagram whose arrows are a subset of the arrows of $D_K$. If $D'$ is a subdiagram of $D_K$ we write $D'<D_K$. Define a formal sum of diagrams:
\[
I(D_{K})=\sum_{D' < D_{K}} D'.
\]
A pairing $\left< \cdot,\cdot\right>$  is defined on Gauss diagrams by $\left<X,Y\right>=1$ if $X=Y$ and $\left<X,Y\right>=0$ otherwise. The pairing is extended by linearity to the free abelian group generated by Gauss diagrams. Then the long virtual knot invariants $v_{2,1}, v_{2,2}$ are given by:
\[
v_{2,1}(K) = \left< \begin{array}{c}
\def\svgwidth{1in}
\begingroup%
  \makeatletter%
  \providecommand\color[2][]{%
    \errmessage{(Inkscape) Color is used for the text in Inkscape, but the package 'color.sty' is not loaded}%
    \renewcommand\color[2][]{}%
  }%
  \providecommand\transparent[1]{%
    \errmessage{(Inkscape) Transparency is used (non-zero) for the text in Inkscape, but the package 'transparent.sty' is not loaded}%
    \renewcommand\transparent[1]{}%
  }%
  \providecommand\rotatebox[2]{#2}%
  \ifx\svgwidth\undefined%
    \setlength{\unitlength}{150.62436917bp}%
    \ifx\svgscale\undefined%
      \relax%
    \else%
      \setlength{\unitlength}{\unitlength * \real{\svgscale}}%
    \fi%
  \else%
    \setlength{\unitlength}{\svgwidth}%
  \fi%
  \global\let\svgwidth\undefined%
  \global\let\svgscale\undefined%
  \makeatother%
  \begin{picture}(1,0.27951749)%
    \put(0,0){\includegraphics[width=\unitlength]{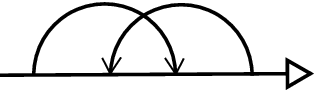}}%
  \end{picture}%
\endgroup%
 \end{array}, I(D_{K})\right>, \,\,\,
v_{2,2}(K) = \left<\begin{array}{c}
\def\svgwidth{1in}
\begingroup%
  \makeatletter%
  \providecommand\color[2][]{%
    \errmessage{(Inkscape) Color is used for the text in Inkscape, but the package 'color.sty' is not loaded}%
    \renewcommand\color[2][]{}%
  }%
  \providecommand\transparent[1]{%
    \errmessage{(Inkscape) Transparency is used (non-zero) for the text in Inkscape, but the package 'transparent.sty' is not loaded}%
    \renewcommand\transparent[1]{}%
  }%
  \providecommand\rotatebox[2]{#2}%
  \ifx\svgwidth\undefined%
    \setlength{\unitlength}{150.62436917bp}%
    \ifx\svgscale\undefined%
      \relax%
    \else%
      \setlength{\unitlength}{\unitlength * \real{\svgscale}}%
    \fi%
  \else%
    \setlength{\unitlength}{\svgwidth}%
  \fi%
  \global\let\svgwidth\undefined%
  \global\let\svgscale\undefined%
  \makeatother%
  \begin{picture}(1,0.27951749)%
    \put(0,0){\includegraphics[width=\unitlength]{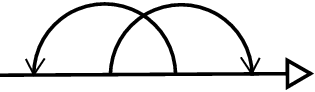}}%
  \end{picture}%
\endgroup%
 \end{array} , I(D_{K})\right>. 
\]
Note here the convention of \cite{GPV} is used, so that if no crossing signs appear on the arrows of a diagram, then the diagram represents the weighted sum of all possible ways the diagram may be signed and the weight of each summand is the product of the signs of all its arrows. For example, the diagram for $v_{2,1}$ represents four diagrams. Each has a pair of arrows as depicted and the signs vary as $(\oplus,\oplus)$, $(\oplus,\ominus)$, $(\ominus,\oplus)$, and $(\ominus,\ominus)$.

\begin{figure}[htb]
\begin{tabular}{|c|c|} \hline &  \\
\begin{tabular}{c} \def\svgwidth{2.1in}
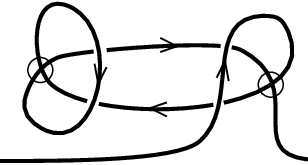 \end{tabular} & \begin{tabular}{c} \def\svgwidth{2.2in}
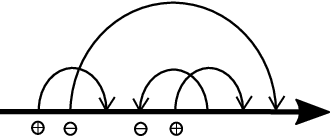 \end{tabular}\\ \hline \end{tabular} 
\caption{(Left) The long virtual knot $L$ shows $v_{2,1} \pmod{2}$ and $v_{2,2}\pmod{2}$ and are not band-pass invariant. (Right) A Gauss diagram of $L$.} \label{fig_fly}
\end{figure}

The Arf invariant of a classical knot $K$ is the coefficient $c_2(K)$ of $z^2$ in the Conway polynomial of $K$, taken modulo two.  On the other hand, we have that $c_2(K)=v_{2,1}(K)=v_{2,2}(K)$. This follows from the fact that the dimension of the the degree two finite-type invariants of classical knots is $1$ and all three invariants have the same value on the right trefoil. Thus the Arf invariant of a classical knot is given by $v_{2,1} (K) \pmod 2 \equiv v_{2,2} (K) \pmod 2$.
\newline
\newline
Although the Arf invariant is defined for long virtual knots, it is neither a band-pass invariant nor a concordance invariant. Let $F$ be the ``fly'' depicted in Figure \ref{fig_conc_ex}.  Its Gauss diagram coincides with the defining Gauss diagram of $v_{2,2}$. Thus, $v_{2,2}(F) \equiv 1 \pmod{2}$ whereas $v_{2,2}$ vanishes on the unknot. Since $F \asymp 1$, $v_{2,2} \pmod{2}$ is not a concordance invariant. Now perform the band-pass move on the long virtual knot $L$ on the right in Figure \ref{fig_fly} to obtain a diagram $R$. Then  $v_{2,2}(L) \pmod{2} \not \equiv v_{2,2}(R) \pmod{2}$. Similarly, $v_{2,1} \pmod{2}$ is neither a band-pass nor concordance invariant. Also note that $v_{2,1}(F)+v_{2,2}(F) \equiv 1 \pmod 2$. Thus, $v_{2,1}+v_{2,2} \pmod 2$ is not a concordance invariant.

\begin{theorem} If $K_1$ and $K_2$ are obtained from one another by a sequence of extended Reidemeister moves and band pass moves, then:
\begin{eqnarray} \label{eqn}
v_{2,1}(K_1)+v_{2,2}(K_1) &\equiv& v_{2,1}(K_2)+v_{2,2}(K_2)  \pmod{2}.
\end{eqnarray}
\end{theorem}
\begin{proof} It suffices to prove the theorem when $K_1$ and $K_2$ differ by a single band pass move. Let $D_1, D_2$ represent Gauss diagrams for the long virtual knots on the left, right, respectively, of a band-pass move in Figure \ref{fig_band_pass}. $D_1$ and $D_2$ differ in a set $\Delta=\{\alpha,\beta,\gamma,\delta\}$ containing exactly four arrows. Here we use the same name for the arrow in $D_1$ and the arrow in $D_2$ obtained by toggling the direction and sign of that arrow. 
\newline
\newline
Suppose $D_1'<D_1$  is any pair of intersecting arrows. $D_1'$ can contain none, one, or two arrows from $\Delta$. If none, then $D_1'$ contributes equally to both sides of Equation \ref{eqn}. Contributions for the other cases depend upon how the arrows of $\Delta$ are configured in the Gauss diagram. The two band-pass moves in Figure \ref{fig_band_pass} have five essentially different ways in which the depicted arc ends may be connected to form a single component, up to change of orientation. The number of cases may be reduced from five to three by noting that an $\Omega 2$ move converts the fourth and fifth configurations into the first and second, respectively. For each configuration, the point at infinity for $K_1$ may be chosen on any of the four dashed connecting arcs. Figure \ref{fig_config} shows one choice; other choices follow similarly.  
\newline
\newline
Consider all subdiagrams of $D_1$ and $D_2$ containing exactly two arrows from $\Delta$. Consulting Figure \ref{fig_config}, one can verify that the total contribution of such subdiagrams to Equation \ref{eqn} is the same on both sides. For example, in the top configuration the pairs $\alpha,\gamma$ and $\gamma,\delta$ each contribute 1 to each side of Equation \ref{eqn} to give a total contribution of $0$ on each side. 
\newline
\newline
Lastly, suppose that $D_1'<D_1$ contains one arrow $x$ from $\Delta$ and one arrow $\eta \not \in \Delta$ that intersects $x$. To create the corresponding subdiagram $D_2'$ in $D_2$, simply change the direction and sign of $x$. Note that the endpoints of $\eta$ lie in separate dashed regions of the real line in Figure \ref{fig_config}. For each $x \in \Delta$, observe in Figure \ref{fig_config} that there is some $y \in \Delta$, $y \ne x$, such that $\eta$ intersects $y$. Now, either $x$ and $y$ point in the same direction or in opposite directions in $D_1$. If they point in the same direction, the pairs $x,\eta$ and $y,\eta$ each contribute 1 to one side of Equation \ref{eqn} and both contribute zero to the other side. If $x$ and $y$ point in the opposite direction, one of the pairs $x,\eta$ and $y,\eta$ contributes 1 on the left of Equation \ref{eqn} while the other pair contributes 1 on the right.
\end{proof}

\begin{figure}[htb]
\begin{tabular}{|c|cc|} \hline
 & & \\
\begin{tabular}{c}
\def\svgwidth{1in}
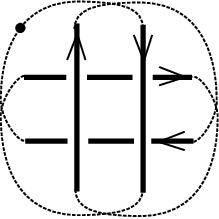 \end{tabular}  &\begin{tabular}{c} \def\svgwidth{1.5in}
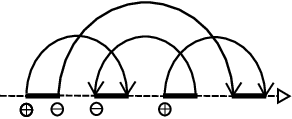 \end{tabular} & \begin{tabular}{c} \def\svgwidth{1.5in}
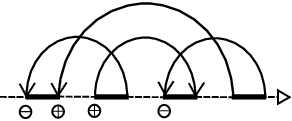 \end{tabular} \\
\begin{tabular}{c}
\def\svgwidth{1in}
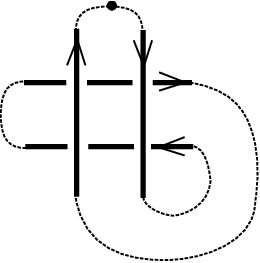 \end{tabular}  & \begin{tabular}{c} \def\svgwidth{1.5in}
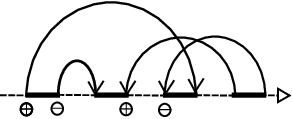 \end{tabular} & \begin{tabular}{c} \def\svgwidth{1.5in}
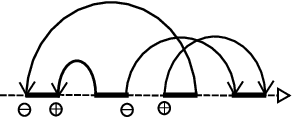 \end{tabular}\\
\begin{tabular}{c}
\def\svgwidth{1in}
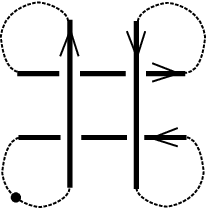 \end{tabular}  & \begin{tabular}{c} \def\svgwidth{1.5in}
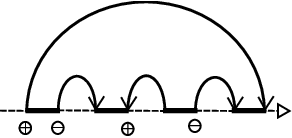 \end{tabular} & \begin{tabular}{c} \def\svgwidth{1.5in}
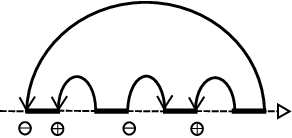 \end{tabular}\\ \hline  & & \\ 
\begin{tabular}{c}
\def\svgwidth{1in}
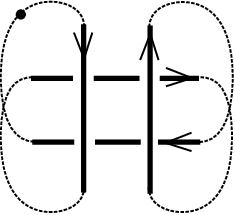 \end{tabular}  & \begin{tabular}{c} \def\svgwidth{1.5in}
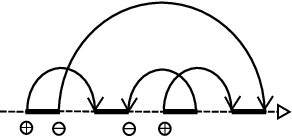 \end{tabular} & \begin{tabular}{c} \def\svgwidth{1.5in}
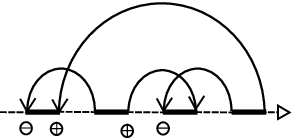 \end{tabular} \\
\begin{tabular}{c}
\def\svgwidth{1in}
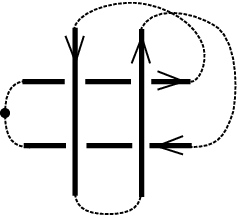 \end{tabular} & \begin{tabular}{c} \def\svgwidth{1.5in}
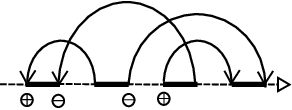 \end{tabular} & \begin{tabular}{c} \def\svgwidth{1.5in}
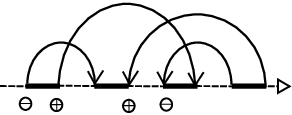 \end{tabular} \\ \hline
\end{tabular}
\caption{The five possible configurations of arcs, and the corresponding Gauss diagrams for the left and right hand sides of the band-pass move.} \label{fig_config}
\end{figure}

\begin{theorem}[Main Theorem]\label{thm_main} For every concordance class $[K] \in \mathscr{VC}$, there is a $J \in [K]$ that is not band-pass equivalent to $K$ and an $L \in [K]$ that is not band-pass equivalent to either the long unknot or a long trefoil.
\end{theorem}

\begin{proof} Again let $F$ be as in Figure \ref{fig_fly}, left. Suppose that $v_{21}(K)+v_{2,2}(K) \equiv 1\pmod 2$. Since $v_{2,1}+v_{2,2} \pmod{2}$ vanishes for any classical knot, $L:=K$ is not band-pass equivalent to either a long unknot or a long trefoil. Set $J:=K\#F$. Then $v_{2,1}(J)+v_{2,2}(J) \equiv 0 \pmod{2}$, since the invariant is additive under concatenation. Since $F$ is concordant to the unknot, we have $J=K\# F \asymp K$.  Thus, $J$ is concordant to $K$ but $J$ is not band-pass equivalent to $K$. If $v_{21}(K)+v_{2,2}(K) \equiv 0 \pmod 2$, set $L:=K\#F$ and $J:=K$.
\end{proof}

\section{Further Discussion}  \label{sec_future}

\subsection{The map from long to closed} \label{sec_map} If $D$ is a Gauss diagram on $\mathbb{R}$, it may be converted into a closed Gauss diagram $\overline{D}$ on $\mathbb{S}^1$ via the map $t \to e^{2 i \arctan(t)}$. The arrow endpoints and crossing signs of $D$ are transferred to $\overline{D}$ in this way. Then $D \to \overline{D}$ descends to a surjective map $K \to \overline{K}$ from long virtual knots to virtual knots \cite{GPV}. Definition \ref{defn_concordance} implies that if $K_1 \asymp K_2$, then $\overline{K_1} \asymp \overline{K_2}$. Any invariant of virtual knot concordance is thus an invariant of long virtual knot concordance. 
\newline
\newline
For example, the Henrich-Turaev polynomial $w_K(t)$ is an order one finite-type concordance invariant of virtual knots \cite{vc_2,henrich, turaev_cobordism}. It takes values in the group of Laurent polynomials $\mathbb{Z}[t,t^{-1}]$. The invariant is $0$ for all classical knots. By virtue of the map $K \to \overline{K}$, $w_K(t)$ is also a finite-type concordance invariant of long virtual knots. Moreover, $w_K:\mathscr{VC} \to \mathbb{Z}[t,t^{-1}]$ is a homomorphism: $w_{K_1\#K_2}(t)=w_{K_1}(t)+w_{K_2}(t)$.
 
\begin{proposition} If $K \asymp -r(K)$, then $K$ has order $1$ or $2$ in $\mathscr{VC}$. There are elements of $\mathscr{VC}$ that are not concordant to a classical knot and have infinite order.
\end{proposition}

\begin{proof} The first fact follows from $K \# K \asymp K \# -r(K) \asymp 1$. Suppose that $K^m \asymp 1$ for some natural number $m$. Then $0=w_{K^m}(t)=m \cdot w_{K}(t)$.  Thus any element with non-vanishing $w_K(t)$ must have infinite order. Consider the long virtual knot diagram $K$ having Gauss code $O1(+)O2(+)U1(+)U2(+)$. Since $w_{K}(t)=2t \ne 0$, $K$ cannot be concordant to any classical knot and has infinite order. \end{proof}

\subsection{Directions for future research} Here we have exclusively focused on the band-pass move for long virtual knots. For the closed case, are there any Goussarov-Polyak-Viro (GPV) finite-type invariants that are band-pass invariant? The answer is likely to be ``yes''. A computer search could resolve this question. One could compute the Polyak algebra modulo the image of the band-pass moves under the map $I$. 
\newline
\newline
The Henrich-Turaev polynomial is a finite-type concordance invariant, but it is not a finite-type invariant in the sense of GPV. For the closed or long case, are there any GPV finite-type invariants that are concordance invariants? The finite-type concordance invariants of classical string links are the Milnor invariants \cite{hm} and these can sometimes be expressed by GPV type formulae \cite{polyak_milnor}. An interesting problem is to classify the finite-type concordance invariants of long virtual knots, and more generally, virtual string links.
\newline
\newline
The group structure of $\mathscr{VC}$ is mysterious. Are there finite order elements in $\mathscr{VC}$ that are not concordant to any classical knot? Are there, for example, any elements having order four? The monoid of long virtual knots is not commutative. Is $\mathscr{VC}$ abelian (essentially posed by Turaev \cite{turaev_cobordism})? The algebraic concordance group is a useful tool for studying the classical knot concordance group. Is there something analogous for long virtual knots?

\begin{acknowledgments} The author is grateful for advice and encouragement from H. U. Boden, H. A. Dye, C. D. Frohman, R. Gaudreau, A. Kaestner, L. H. Kauffman, A. Nicas and R. Todd. 
\end{acknowledgments}
\bibliographystyle{plain}
\bibliography{bib_bandpass}

\end{document}